\documentclass{amsproc}
\usepackage{graphicx,mathrsfs,amssymb,amsmath,hyperref}
 \usepackage{amscd}

\newtheorem{theorem}{Theorem}[section]
\newtheorem{lemma}[theorem]{Lemma}
\newtheorem{corollary}[theorem]{Corollary}
\newtheorem{proposition}[theorem]{Proposition}
\newtheorem{remark}[theorem]{Remark}
\newtheorem{definition}[theorem]{Definition}

\theoremstyle{remark}
\newtheorem{example}[theorem]{Example}

\numberwithin{equation}{section}

\newcommand{\bz}{{\mathbb B}}
\newcommand{\cz}{{\mathbb C}}
\newcommand{\gz}{{\mathbb Z}}

\newcommand{\nz}{{\mathbb N}}

\newcommand{\rz}{{\mathbb R}}
\newcommand{\sz}{{\mathbb S}}

\newcommand{\bfg}{\mathbf{g}}

\newcommand{\calA}{\mathcal{A}}
\newcommand{\calB}{\mathcal{B}}

\newcommand{\calF}{\mathcal{F}}

\newcommand{\calP}{\mathcal{P}}
\newcommand{\calR}{\mathcal{R}}

\newcommand{\calT}{\mathcal{T}}

\newcommand{\scrC}{\mathscr{C}}

\newcommand{\scrH}{\mathscr{H}}

\newcommand{\scrL}{\mathscr{L}}
\newcommand{\scrS}{\mathscr{S}}

\newcommand{\cl}{\mathrm{cl}}

\newcommand{\dbar}{d\hspace*{-0.08em}\bar{}\hspace*{0.1em}}
\newcommand{\eps}{\varepsilon}

\newcommand{\forget}[1]{}

\newcommand{\lra}{\longrightarrow}
\newcommand{\op}{\mathrm{op}}

\newcommand{\re}{\mathrm{Re}\,}

\newcommand{\st}{\mbox{\boldmath$\;|\;$\unboldmath}}

\newcommand{\wh}{\widehat}
\newcommand{\wt}{\widetilde}

\begin{document}
\title[Ellipticity in Pseudodifferential Algebras of Toeplitz Type]{
Ellipticity in Pseudodifferential\\ Algebras of Toeplitz Type}

\author{J\"org Seiler}
\address{Dipartimento di Matematica, Universit\`{a} di Torino, Italy}
\email{joerg.seiler@unito.it}

\begin{abstract}
Let $L^\star$ be a filtered algebra of abstract pseudodifferential operators equipped with a notion of ellipticity, and $T^\star$ be a subalgebra of operators of the form $P_1AP_0$, where $P_0$ and $P_1$ are two projections, i.e.,  $P_j^2=P_j$. The elements of $L^\star$ act as linear continuous operators in certain scales of abstract Sobolev spaces, the elements of the subalgebra in the corresponding subspaces determined by the projections. We study how the ellipticity in $L^\star$ descends to $T^\star$, focusing on parametrix construction, Fredholm property, and homogeneous principal symbols. Applications concern $SG$-pseudodifferential operators, pseudodifferential operators on manifolds with conical singularities, and Boutet de Monvel's algebra for boundary value problems. In particular, we derive invertibilty of the Stokes operator with Dirichlet boundary conditions in a subalgebra of Boutet de Monvel's algebra. We indicate how the concept generalizes to parameter-dependent operators.  
\end{abstract}

\maketitle

\tableofcontents

\section{Introduction}\label{sec:intro}

One of the major ideas in the theory of pseudodifferential operators is to study existence and regularity of solutions to partial differential equations in terms of a parametrix construction within an algebra of pseudodifferential operators. For example, given a $\mu$-th order differential operator $A$ on a closed smooth manifold $M$ whose homogeneous principal symbol is pointwise invertible on the unit co-sphere bundle of $M$, one can construct a pseudodifferential operator $B$ of order $-\mu$  such that both $AB-1$ and $BA-1$ are smoothing operators, i.e., are integral operators with a smooth integral kernel. Such an operator $B$ is called a parametrix of $A$. Concerning the partial differential equation $Au=f$ this has two important consequences: 
\begin{itemize}
 \item[$(1)$] \emph{Fredholm property}: If $H^s(M)$ denotes the standard $L_2$-Sobolev space 
  of regularity $s$, the operator $A:H^s(M)\to H^{s-\mu}(M)$ is a Fredholm operator, i.e., 
  its kernel is finite dimensional and its range is finite co-dimensional. Hence for any $f$ 
  satisfying a finite number of orthogonality conditions, the $($affine$)$ solution space is 
  finite dimensional. 
 \item[$(2)$] \emph{Elliptic regularity}: If $f$ has regularity $s-\mu$ then any solution $u$ 
  which has some a-priori regularity $t$ must have regularity $s$. 
\end{itemize}
The first property holds since smoothing operators are compact operators in any $H^s(M)$, the second is due to the fact that $B:H^{s-\mu}(M)\to H^s(M)$ and due to the smoothing property of the remainders. 

This concept -- to embed differential operators in a class of pseudodifferential operators and to construct parametrices to elliptic elements and to obtain the Fredholm property and elliptic regularity of solutions -- has by now been realized for a huge variety of different kinds of differential operators. Just to name a few, let us mention Boutet de Monvel's algebra for boundary value problems \cite{Bout1}, Schulze's calculi for manifolds with conical singularities, edges or higher singularities \cite{Schu91}, \cite{EgSc}, and Melrose's b-calculus  \cite{Melr93} for manifolds with corners. 

Boutet de Monvel's algebra for boundary value problems on a compact manifold $M$ with smooth boundary consists of operators of the form 
 \begin{equation*}
  \calA=
  \begin{pmatrix}
   A_++G & K \\
   T & Q
  \end{pmatrix}\;:\quad
  \begin{matrix}
   H^s(M,E_0)\\
   \oplus\\
   H^s(\partial M,J_0)
  \end{matrix}
  \longrightarrow
  \begin{matrix}
   H^{s-\mu}(M,E_1)\\
   \oplus\\
   H^{s-\mu}(\partial M,J_1)
  \end{matrix},
 \end{equation*}
where $E_j$ and $J_j$ are vector bundles over $M$ and $\partial M$, respectively, which are  allowed to be zero dimensional. Here, $A_+$ is the ``restriction" of a pseudodifferential operator $A$ on the double of $M$ to $M$, $G$ is a so-called singular Green operator, $K$ is a potential operator, $T$ is a trace operator, and $Q$ is a usual pseudodifferential operator on the boundary. $H^s$ refers to the $L_2$-Sobolev spaces. For further information on this calculus we refer the reader to \cite{Bout1}, \cite{Schu37}, or  \cite{Schr01};  in Section \ref{sec:6.1} we shortly sketch some details. For example, the Dirichlet problem for the Laplacian is included in this set-up as  
 $$\begin{pmatrix}
   \Delta \\ T 
  \end{pmatrix}\;:\quad
   H^s(M)
  \longrightarrow
  \begin{matrix}
   H^{s-2}(M)\\
   \oplus\\
   H^{s-2}(\partial M)
  \end{matrix},\qquad s>\frac{1}{2},$$
where $Tu=S(u|_{\partial M})$ with an invertible pseudodifferential operator $S$ of order $3/2$ on the boundary. The parametrix $($which in fact is an inverse in this case$)$ is then of the form $\begin{pmatrix}A_++G & K \end{pmatrix}$, where $u=(A_++G)f$ solves $\Delta u=f$ in $M$ and $u|_{\partial M}=0$, while $u=K\varphi$ solves $\Delta u=0$ in $M$ and $u|_{\partial M}=S^{-1}\varphi$. Ellipticity in Boutet de Monvel's algebra is determined by the invertibility of both the homogeneous principal symbol of $A$ and the principal boundary symbol associated with $\calA$. For differential problems this corresponds to Shapiro-Lopatinskii ellipticity of boundary value problems.  

Many boundary value problems fit into the framework of Boutet de Monvels algebra. However, there are important exceptions. For example, it is known that Dirac operators in even dimension cannot be completed with a boundary condition to be an elliptic element in Boutet de Monvel's algebra. For example, the Dirichlet condition $u\mapsto u|_{\partial M}$ in this case cannot be used directly, but can be replaced by the Atiyah-Patodi-Singer boundary condition $u\mapsto P(u|_{\partial M})$ where $P$ is the positive spectral projection of the tangential operator associated with the Dirac operator. In \cite{Schu37} Schulze extended the Boutet de Monvel calculus in such a way that problems of APS-type are included. In this extended calculus the operators are of the form 
\begin{equation}\label{eq:introB}
  \begin{pmatrix}
   1 & 0 \\
   0 & P_1
  \end{pmatrix}
  \begin{pmatrix}
   A_++G & K \\
   T & Q
  \end{pmatrix}
  \begin{pmatrix}
   1 & 0 \\
   0 & P_0
  \end{pmatrix},
\end{equation}
where $P_j$ are zero-order pseudodifferential operators on the boundary which are projections, i.e., $P_j^2=P_j$. Defining the closed subspaces 
 $$H^s(\partial M,J_j,P_j):=P_j\big(H^s(\partial M,J_j)\big)$$ 
of $H^s(\partial M,J_j)$, operators of the form \eqref{eq:introB} are considered as a maps  
\begin{equation}\label{eq:introA}
  \begin{matrix}
   H^s(M,E_0)\\
   \oplus\\
   H^s(\partial M,J_0,P_0)
  \end{matrix}
  \longrightarrow
  \begin{matrix}
   H^{s-\mu}(M,E_1)\\
   \oplus\\
   H^{s-\mu}(\partial M,J_1,P_1).
  \end{matrix} 
\end{equation}
A concept of ellipticity is developed and it is shown that elliptic operators have a parametrix of analogous structure, i.e., as in \eqref{eq:introB} but with $P_0$ and $P_1$ interchanged. For elliptic operators the map \eqref{eq:introA} is fredholm, and one has elliptic regularity in the scale of projected subspaces. In \cite{SS04} and \cite{SS06} Schulze and the author realized a similar calculus for boundary value problems without transmission property and operators on manifolds with edges, respectively. 

In the present paper we study such kind of calculi with projections from a general point of view. We consider a filtered algebra $L^\star$ of ``abstract" pseudodifferential operators equipped with a notion of ellipticity that is equivalent to the existence of a parametrix modulo smoothing operators, cf.\ Section \ref{sec:2} for details. Let us denote by $T^\star$ the subalgebra of operators of the form $P_1AP_0$, where $A$ belongs to $L^\star$ and $P_0$, $P_1$ are projections, i.e., $P_j^2=P_j$. In view of the structure of classical Toeplitz operators on the unit-circle\footnote{which are of the form $PM_fP$ where $P$ is the orthogonal projection of $L_2(S^1)$ onto the Hardy space of functions that extend holomorphically to the unit-disc, and $M_f$ denotes the operator of multiplication by a bounded function $f$ on the unit-circle} we call $T^\star$ a \emph{Toeplitz subalgebra}. We shall assume that the elements of $L^\star$ act as linear continuous operators in certain scales of ``abstract" Sobolev spaces. Applying the projections to these spaces yields a natural scale of closed subspaces in which the elements of the Toeplitz subalgebra act. Recall that the algebra of operators from \eqref{eq:introB} fits in this general framework. In Section \ref{sec:3} we show that ellipticity in $L^\star$ naturally induces a concept of ellipticity in the Toeplitz subalgebra $T^\star$. Assuming that $L^\star$ is closed under taking adjoints and that ellipticity is equivalent to the Fredholm mapping property, we show that also the induced ellipticity in $T^\star$ is equivalent to the Fredholm mapping property. Moreover, if ellipticity in $L^*$ is characterized by the invertibility of certain ``abstract" principal symbols, we show how the corresponding symbolic structure looks like in the Toeplitz subalgebra. Also we give a sufficient condition insuring the spectral invariance of Toeplitz subalgebras. In Section \ref{sec:3.5} we indicate how our approach can be extended to parameter-dependent operators; this will be further developed in a separate article. 

In the last three sections we discuss various concrete examples that are all covered by our approach. In particular, the above described generalized Boutet de Monvel calculus is obtained as a particular case of our general considerations as we shall discuss in Section \ref{sec:6}. The other examples concern $SG$-pseudodifferential operators on $\rz^n$, cf.\ \cite{Pare}, \cite{Cord}, \cite{Schr87}, as well as pseudodifferential operators on manifolds with conical singularities. For the latter we work with Schulze's cone algebra \cite{Schu91}. These examples are discussed in Section \ref{sec:4} and Section \ref{sec:5}, respectively. 

An important role in the analysis of the Navier-Stokes equations plays the so-called Stokes operator, i.e., the Laplacian considered on divergence free $($solenoidal$)$ vector fields. More precisely, if $M$ is a smoothly bounded compact domain in $\rz^n$ and $P$ denotes the Helmholtz projection for $M$ then the Stokes operator with Dirichlet boundary conditions is  
\begin{equation}\label{eq:st}
 P\Delta: H^s_\sigma(M,\cz^n)\cap\{u\st u|_{\partial M}=0\}\lra H^{s-2}_\sigma(M,\cz^n),\qquad s>3/2,
\end{equation}
where $H^s_\sigma(M,\cz^n):=P(H^s(M,\cz^n))$ with the usual $L_2$-Sobolev space $H^s(M,\cz^n)$ of smoothness $s$. In \cite{Giga} Giga has shown that in case $s=2$ this operator $($even in correponding $L_p$-Sobolev spaces$)$ is the generator of an analytic semigroup and that its resolvent has a certain pseudodifferential structure. Grubb \cite{GrSo}, \cite{Grub} studied the resolvent in terms of a parameter-dependent Boutet de Monvel algebra. In our context of Toeplitz subalgebras, we can view the Stokes operator as an element in a Toeplitz subalgebra of Boutet de Monvel's algebra. Taking for granted that \eqref{eq:st} is an isomorphism for $s=2$ (in fact, it is a self-adjoint positive operator) we will derive in Section \ref{sec:6.4} that it is an isomorphism for any $s>3/2$ and we shall show that the inverse is of the form  $P(A_++G)P$, where $A_++G$ belongs to Boutet de Monvel's algebra. Though this result might not be new, it follows without effort as a straightforward  corollary from our general approach. 

\section{The general set-up}\label{sec:2}

In this section we describe algebras of ``abstract pseudodifferential operators" and subalgebras of Toeplitz type. In this general set-up we incorporate standard features frequently met in pseudodifferential analysis: ellipticity, parametrix construction, Fredholm property, and homogeneous principal symbols. 

\subsection{Algebras of generalized pseudodifferential operators}\label{sec:2.1}

Let $\nz_0$ denote the set of non-negative integers. 
We shall consider a set of operators 
\begin{equation}\label{eq:pseudo}
 L^\star=\mathop{\mbox{\large$\cup$}}_{-\mu\in\nz_0}
 \mathop{\mbox{\large$\cup$}}_{\bfg\in\mathbf{G}}
 L^\mu(\bfg),
\end{equation}
where $\mu$ represents the ``order of operators" while $\mathbf{G}$ is a set of 
``admissible" pairs of data $\bfg=(g_0,g_1)$; at this stage $\mathbf{G}$ should be considered 
as data specifying the class of considered operators, and to which will be assigned a precise 
meaning depending on the concrete application. For a first example see Example \ref{ex:first} 
below.  

We assume that with any single datum $g$ there is associated a scale of Hilbert spaces 
\begin{equation}\label{eq:sobolev}
 H^s(g),\qquad s\in\nz_0,
\end{equation}
and that each element $A\in L^\mu(\bfg)$, $\bfg=(g_0,g_1)$, induces continuous 
linear operators 
 $$A:H^s(g_0)\lra H^{s-\mu}(g_1)$$
for any $s\ge0$. Furthermore we ask that 
 $$L^\mu(\bfg)\subset L^\nu(\bfg),\qquad\mu\le\nu,$$
and that for two data $\bfg=(g_0,g_1)$ and $\bfg^\prime=(g_1,g_2)$ 
composition of operators induces mappings 
\begin{equation}\label{eq:mult}
 L^\mu(\bfg^\prime)\times L^\nu(\bfg)\lra L^{\mu+\nu}(\bfg^\prime\circ\bfg),
 \qquad \bfg^\prime\circ\bfg:=(g_0,g_2).
\end{equation}
Due to this composition property we also speak -- by abuse of language -- of the ``algebras" 
$L^\star$ or $L^\mu(\bfg)$. The classes of ``smoothing" operators are defined as
 $$L^{-\infty}(\bfg):=\mathop{\mbox{\large$\cap$}}_{-\mu\in\nz_0}L^\mu(\bfg).$$

\begin{example}\label{ex:first}
Let $M$ be a smooth compact $($Riemannian$)$ manifold. 
A ``datum" is any pair $g=(M,E)$ where $E$ is a smooth $($hermitean$)$  
vector bundle over $M$. We let 
 $$H^s(g)=H^s(M,E)$$
denote the standard $L_2$-Sobolev spaces of sections into $E$ of regularity $s$. 
For $\bfg=(g_0,g_1)$ with $g_j=(M,E_j)$ let 
 $$L^\mu(\bfg)=L^\mu_\cl(M;E_0,E_1)$$
be the space of classical pseudodifferential operators, mapping sections into $E_0$ to sections 
into $E_1$.\footnote{With ``classical" we mean that the local pseudodifferential symbols have 
complete asymptotic expansions into homogeneous components; see Section \ref{sec:4} for 
details.} 
\end{example}

A standard feature in pseudodifferential analysis is the possibility of asymptotic summation. 
We incorporate this by the following definition. 

\begin{definition}
We call the algebra $L^\star$ \emph{asymptotically complete} if for any admissible $\bfg$ and any  
sequence of operators $A_j\in L^{-j}(\bfg)$ there exists an operator $A\in L^0(\bfg)$ such that 
$A-\sum\limits_{j=0}^{N-1}A_j\in L^{-N}(\bfg)$ for any positive integer $N$.   
\end{definition} 

Next we come to ellipticity and parametrices. 

\begin{definition}
An operator $A\in L^0(\bfg)$, $\bfg=(g_0,g_1)$, is called elliptic if there exists a 
$B\in L^{0}(\bfg^{(-1)})$ with $\bfg^{(-1)}:=(g_1,g_0)$ such that 
 $$BA-1\in L^{-\infty}(\bfg_0),\qquad AB-1\in L^{-\infty}(\bfg_1),$$
where $\bfg_0=(g_0,g_0)$, $\bfg_1=(g_1,g_1)$. 
Any such operator $B$ is called a parametrix of $A$. 
\end{definition}

Obviously parametrices are uniquely determined modulo smoothing operators. 
As a consequence of a von Neumann series argument the existence of a parametrix 
in case of asymptotic completeness is equivalent to the existence of left- and right-inverses 
modulo operators of order $-1$: The operator $A\in L^\mu(\bfg)$ is elliptic if there exist 
$B_0,B_1\in L^{-\mu}(\bfg^{(-1)})$ such that 
$B_0A-1\in L^{-1}(\bfg_0)$ and $AB_1-1\in L^{-1}(\bfg_1)$. For details see the proof of 
Proposition \ref{prop:complete}, below. 

The fact that parametrices are inverses modulo smoothing operators implies elliptic regularity of associated equations: If $A\in L^0(\bfg)$ is elliptic, $f\in H^s(g_1)$, and $u\in H^0(g_0)$ then $Au=f$ implies $u\in H^s(g_0)$. In applications one is also interested in the fact that the smoothing remainders yield compact operators in the associated spaces, since this implies that elliptic operators are Fredholm operators. Even more, one also wants that the Fredholm property of an operator implies its ellipticity, meaning that the notion of ellipticity is actually optimal. 

\begin{definition}
We say that $L^\star$ has the \emph{Fredholm property} if, for any admissible $\bfg$, 
the following holds: 
\begin{itemize}
 \item[$($a$)$] Any $R\in L^{-\infty}(\bfg)$ is a compact operator 
  $H^0(g_0)\to H^0(g_1)$.
 \item[$($b$)$] If $A\in L^{0}(\bfg)$ is a Fredholm operator 
  $H^0(g_0)\to H^0(g_1)$ then $A$ is elliptic. 
\end{itemize}
\end{definition}

The algebra of classical pseudodifferential operators on a compact manifold, cf.\ Example \ref{ex:first} is asymptotically complete and has the Fredholm property in the sense of the previous definitions. 

\subsection{Toeplitz subalgebras}\label{sec:2.2}

In the following let $\bfg=(g_0,g_1)$ be a weight-datum and $\bfg_0=(g_0,g_0)$, 
$\bfg_1=(g_1,g_1)$. 

\begin{definition}
Let $P_0\in L^0(\bfg_0)$ and $P_1\in L^0(\bfg_1)$ be two projections $($i.e.,  
$P_j^2=P_j)$. Then we define 
 $$T^\mu(\bfg,P_0,P_1):=\Big\{A\in L^\mu(\bfg)\st (1-P_1)A=0,\;A(1-P_0)=0\Big\}$$
and we set 
\begin{equation*}
 T^\star=\mathop{\mbox{\large$\cup$}}_{-\mu\in\nz_0}
 \mathop{\mbox{\large$\cup$}}_{\substack{\bfg\in\mathbf{G},\\ P_j\in L^0(\bfg_j)}}
 T^\mu(\bfg,P_0,P_1). 
\end{equation*} 
\end{definition}

In analogy to classical Toeplitz operators on the circle we call 
$T^\mu(\bfg,P_0,P_1)$ a \emph{Toeplitz subalgebra} of $L^\mu(\bfg)$. We let 
\begin{equation*}
 H^s(g_j,P_j)=P_j\big(H^s(g_j)\big),\qquad s\in\nz_0.
\end{equation*}
Note that $H^s(g_j,P_j)$ is a closed subspace of $H^s(g_j)$ and that any element 
$A\in T^\mu(\bfg,P_0,P_1)$ induces continuous operators 
 $$A:\;H^s(g_0,P_0)\lra H^{s-\mu}(g_1,P_1).$$ 
Observe that the canonical map 
 $$\wt{A}\mapsto P_1\wt{A}P_0:\;L^\mu(\bfg)\lra T^\mu(\bfg,P_0,P_1)$$
is surjective; in other words, we can write 
\begin{equation}\label{eq:talg}
 T^\mu(\bfg,P_0,P_1)=P_1\,L^\mu(\bfg)\,P_0. 
\end{equation}

\begin{definition}\label{def:ell}
An operator $A\in T^0(\bfg,P_0,P_1)$ is called elliptic if there exists a 
$B\in T^{0}(\bfg^{(-1)},P_1,P_0)$ such that 
 $$BA-P_0\in T^{-\infty}(\bfg_0,P_0,P_0),\qquad AB-P_1\in T^{-\infty}(\bfg_1,P_1,P_1).$$
Any such operator $B$ is called a parametrix of $A$. 
\end{definition}

Note that $P_j$ is the identity operator on $H^s(g_j,P_j)$. Therefore ellipticity in fact asks 
for the existence of $B\in T^{0}(\bfg^{(-1)},P_1,P_0)$ and $R_j\in T^{-\infty}(\bfg_j,P_j,P_j)$ 
such that $R_0=BA-1$ on $H^s(g_0,P_0)$ and $R_1=AB-1$ on $H^s(g_1,P_1)$. 

\begin{remark}
Referring to the previously used notation we could introduce new weight-data  
 $$\wh\bfg=(\wh{g}_0,\wh{g}_1):=((g_0,P_0),(g_1,P_1)),\qquad 
     P_j\in L^0(\bfg_j)\text{ projection},$$
and then write $L^\mu(\wh\bfg):=T^\mu(\bfg,P_0,P_1)$. Thus we could use for Toeplitz 
algebras the same formalism as above. However, we find it more intuitive to use the 
notation $T^\mu(\bfg,P_0,P_1)$ in the sequel.
\end{remark}

\section{Ellipticity in Toeplitz subalgebras}\label{sec:3}

We shall investigate how the notion of ellipticity in the full algebra   
$L^0(\bfg)$ descends to a Toeplitz subalgebra $T^0(\bfg,P_0,P_1)$. 
We shall use the notation introduced in Section \ref{sec:2}. 

\subsection{Asymptotic summation and parametrices}\label{sec:3.1}

A first simple observation is that asymptotic completeness passes over to 
Toeplitz subalgebras: 

\begin{lemma}\label{lem:complete}
If $L^\star$ is asymptotically complete then so is $T^\star$.
\end{lemma}
\begin{proof}
Let $A_j\in T^{-j}(\bfg,P_0,P_1)$ be a given sequence of operators. 
Then there exists an $\wt{A}\in L^0(\bfg)$ with 
$\wt{A}-\sum\limits_{j=0}^{N-1}A_j\in L^{-N}(\bfg)$
for any $N$. Choosing $A=P_1\wt{A}P_0$ 
we have ${A}-\sum\limits_{j=0}^{N-1}A_j\in T^{-N}(\bfg,P_0,P_1)$,
since $A_j=P_1A_jP_0$ for any $j$. 
\end{proof}

\begin{proposition}\label{prop:complete}
Let $L^\star$ be asymptotically complete. Then for $A\in T^0(\bfg,P_0,P_1)$ 
the following statements are equivalent:
\begin{itemize}
 \item[$(1)$] $A$ is elliptic. 
 \item[$(2)$] There exist $B_0,B_1\in T^{0}(\bfg^{(-1)},P_1,P_0)$ such that 
  $B_0A-P_0\in T^{-1}(\bfg_0,P_0,P_0)$ and $AB_1-P_1\in T^{-1}(\bfg_1,P_1,P_1)$. 
 \item[$(3)$] There exist $C_0,C_1\in L^{0}(\bfg^{(-1)})$ such that 
  $C_0A-P_0\in L^{-1}(\bfg_0)$ and $AC_1-P_1\in L^{-1}(\bfg_1)$. 
\end{itemize}
\end{proposition}
\begin{proof}
The implications $(1)\Rightarrow (2)$ and $(2)\Rightarrow (3)$ are obvious. 

Let us now show that $(3)$ implies $(2)$. If $C_0A-P_0=R_0$ with $R_0\in L^{-1}(\bfg_0)$ 
then multiplication from the left and the right with $P_0$ and the fact that $P_1A=A=AP_0$ 
yield $P_0C_0P_1A-P_0=P_0R_0P_0$. Similarly we get $AP_0C_1P_1-P_1=P_1R_1P_1$ with 
$R_1\in L^{-1}(\bfg_1)$. Thus $(2)$ holds with $B_j=P_0C_jP_1$ for $j=0,1$. 

Finally, assume $(2)$ is true. Hence there is an $R_0\in T^{-1}(\bfg_0,P_0,P_0)$ with 
 $$B_0A=P_0-R_0=(P_0-R_0)P_0.$$ 
Since $T^*$ is asymptotically complete by Lemma \ref{lem:complete}, we can choose an element  
$B_0^\prime\in T^0(\bfg_0,P_0,P_0)$ with 
$B_0^\prime\sim P_0+\sum\limits_{\ell=1}^\infty R_0^\ell$. Then $B_L:=B_0^\prime B_0$ is a 
left-parametrix of $A$, i.e. $B_LA-P_0\in T^{-\infty}(\bfg_0,P_0,P_0)$. Analogously we can 
construct a right-parametrix $B_R$ and then we can choose $B=B_L$ or $B=B_R$. 
\end{proof}

As an immediate consequence of part $(3)$ of the previous Proposition we obtain: 

\begin{corollary}
Let $L^\star$ be asymptotically complete, $\wt{A}\in L^0(\bfg)$ and $Q_j\in L^0(\bfg_j)$ be two 
projections with $P_j-Q_j\in L^{-1}(\bfg_j)$. Then 
$P_1\wt{A}P_0\in T^\mu(\bfg,P_0,P_1)$ is elliptic if, and only if, 
$Q_1\wt{A}Q_0\in T^\mu(\bfg,Q_0,Q_1)$ is elliptic. 
\end{corollary} 

\subsection{The Fredholm property}\label{sec:3.2}

Assume that $L^0(\bfg)$ has the Fredholm property. Since 
$H^0(g_j,P_j)$ is a closed subspace of $H^0(g_j)$ it is clear that smoothing operators from 
$T^{-\infty}(\bfg,P_0,P_1)$ induce compact operators $H^0(g_0,P_0)\to H^0(g_1,P_1)$. 
Therefore elliptic elements from $T^{0}(\bfg,P_0,P_1)$ induce Fredholm operators 
$H^0(g_0,P_0)\to H^0(g_1,P_1)$. For the reverse statement we shall need that the algebras 
are stable under taking adjoints. 

\begin{definition}
We call $L^\star$ \emph{$*$-closed} if for any admissible $\bfg$ and any $A\in L^0(\bfg)$ 
there exists an $A^*\in L^0(\bfg^{(-1)})$ such that $A^*:H^0(g_1)\to H^0(g_0)$ coincides 
with the adjoint of $A:H^0(g_0)\to H^0(g_1)$. 
\end{definition}

Since any $A\in T^0(\bfg,P_0,P_1)$ satisfies $A=P_1AP_0$, taking the adjoint 
in the $*$-closed algebra $L^0(\bfg)$ yields a map 
 $$T^0(\bfg,P_0,P_1)\lra T^0(\bfg^{(-1)},P_1^*,P_0^*).$$
This map preserves Fredholm operators: 

\begin{lemma}\label{lem:adjoint}
Let $A\in T^0(\bfg,P_0,P_1)$ induce a Fredholm operator $A:H^0(g_0,P_0)\to H^0(g_1,P_1)$. 
Then $A^*:H^0(g_1,P_1^*)\to H^0(g_0,P_0^*)$ is also a Fredholm operator. 
\end{lemma}
\begin{proof}
First observe that there is a natural identification of the dual space of $H^0(g_j,P_j)$ with 
$H^0(g_j,P_j^*)$. In fact, any functional $x$ in the dual space $H^0(g_j,P_j)^\prime$ can be 
extended to one in $H^0(g_j)^\prime$ by setting 
 $$\wt{x}(u)=x(P_ju),\qquad u\in H^0(g_j).$$ 
If we denote by $I_j:H^0(g_j)\to H^0(g_j)^\prime$ the standard Riesz isomorphism then 
 $$x\mapsto P_j^*I_j^{-1}\wt{x}:\;H^0(g_j,P_j)^\prime\lra H^0(g_j,P_j^*)$$
is an isomorphism. Under this identification the dual operator 
$A^\prime:H^0(g_1,P_1)^\prime\to H^0(g_0,P_0)^\prime$ corresponds to 
$A^*:H^0(g_1,P_1^*)\to H^0(g_0,P_0^*)$. Now it remains to observe that duals of Fredholm
operators remain being Fredholm operators. 
\end{proof}


\begin{lemma}\label{lem:fred}
Let $X$ and $Y$ be two Hilbert spaces and $T:X\to Y$ be an upper semi-fredholm 
operator, i.e., $T$ has closed range and a finite-dimensional kernel.  
Then $T^*T$ is a Fredholm operator. 
\end{lemma} 
\begin{proof}
$T^*T$ has finite-dimensional kernel, since $\mathrm{ker}\,T^*T=\mathrm{ker}\,T$. 
Since the range of $T$ is closed, we have the orthogonal decomposition 
$Y=\mathrm{im}\,T\oplus \mathrm{ker}\,T^*$. Therefore $\mathrm{im}\,T^*T=\mathrm{im}\,T^*$. 
Since $T^*$ is lower semi-fredholm, its range has finite co-dimension. 
\end{proof}

\begin{theorem}\label{thm:fred}
If $L^\star$ has the Fredholm property and is $*$-closed, 
then $T^\star$ has the Fredholm property. 
\end{theorem}
\begin{proof}
Let $A\in T^0(\bfg,P_0,P_1)$ induce a Fredholm operator 
$H^0(g_0,P_0)\to H^0(g_1,P_1)$; let us denote this operator by $\wh{A}$. 
We have to show that $A$ is elliptic. 
Due to the $*$-closedness of $L^\star$ we have
\begin{equation}\label{eq:opB} 
 B:=A^*A+(1-P_0)^*(1-P_0)\in L^0(\bfg_0).
\end{equation}
We shall now show that $B:H^0(g_0)\to H^0(g_0)$ is a Fredholm operator. To this end define 
 $$T:H^0(g_0)\lra H^0(g_1)\oplus H^0(g_0),\qquad 
   Tu=\big(Au,(1-P_0)u\big).$$
Now $\mathrm{ker}\,T=\mathrm{ker}\,\wh{A}$, and 
$\mathrm{im}\,T=\mathrm{im}\,\wh{A}\oplus \mathrm{im}\,(1-P_0)$ is closed in 
$H^0(g_1)\oplus H^0(g_0)$, i.e., $T$ is upper semi-fredholm. 
Due to Lemma \ref{lem:fred}, $B=T^*T$ is fredholm. 

Since $L^\star$ has the Fredholm property, $B$ has a parametrix $C\in L^0(\bfg_0)$, i.e., 
$R:=CB-1\in L^{-\infty}(\bfg_0)$. Therefore  
\begin{align*}
 P_0RP_0=P_0C\big(A^*A+(1-P_0)^*(1-P_0)\big)P_0-P_0=P_0CA^*P_1A-P_0,
\end{align*}
since $P_1A=AP_0=A$. Therefore $B_L:=P_0CA^*P_1\in T^0(\bfg^{(-1)},P_1,P_0)$ is a left-parametrix 
of $A$. 

In view of Lemma \ref{lem:adjoint} we can construct in the same way a left-parametrix to 
$A^*\in T^0(\bfg^{(-1)},P_1^*,P_0^*)$. The adjoint of this left-parametrix yields a 
right-parametrix for $A$. Hence $A$ is elliptic. 
\end{proof}

\begin{remark}\label{rem:opB}
For later purpose let us state here that the operator $B$ defined in \eqref{eq:opB} is a 
self-adjoint Fredholm operator with 
 $$\mathrm{ker}\,\big(B:H^0(g_0)\to H^0(g_0)\big)=
   \mathrm{ker}\,\big(A:H^0(g_0,P_0)\to H^0(g_1,P_1)\big).$$
In fact, the kernel on the right-hand side is clearly contained in the kernel on the left-hand 
side. Moreover, $Bu=0$ implies  
 $$0=(Bu,u)=(A^*Au,u)+((1-P_0)^*(1-P_0)u,u)=\|Au\|^2+\|(1-P_0)u\|^2,$$
where inner-product and norm are those of $H^0(g_0)$. Thus $Au=0$ and $(1-P_0)u=0$. 
In particular, $B$ is an isomorphism if $A:H^0(g_0,P_0)\to H^0(g_1,P_1)$ is injective. 
\end{remark}

Another interesting property of many pseudodifferential calculi is their ``spectral invariance", i.e., whenever an element of the algebra is invertible as a continuous operator between Sobolev spaces, then the inverse belongs to the calculus.  

\begin{theorem}\label{thm:spectral}
Let $L^\star$ have the Fredholm property and be $*$-closed. Furthermore assume that $R_1TR_0\in L^{-\infty}(\bfg)$ whenever $\bfg$ is admissible, $R_j\in L^{-\infty}(\bfg_j)$ are smoothing operators, and $T:H^0(g_0)\to H^0(g_1)$ continuously. Then $T^0(\bfg,P_0,P_1)$ is spectrally invariant for any admissible $\bfg$, i.e., if $A\in T^0(\bfg,P_0,P_1)$ induces an isomorphism  $H^0(g_0,P_0)\to H^0(g_1,P_1)$ then there exists a $B\in T^0(\bfg^{(-1)},P_1,P_0)$ such that $BA=P_0$ and $AB=P_1$.  
\end{theorem}
\begin{proof}
Let $A$ be as stated. In particular, $A:H^0(g_0,P_0)\to H^0(g_1,P_1)$ is a Fredholm operator. By Theorem \ref{thm:fred} there exists a parametrix $B\in T^0(\bfg^{(-1)},P_1,P_0)$. In particular, $BA=1-R_0$ on $H^0(g_0,P_0)$ and $AB=1-R_1$ on $H^0(g_1,P_1)$ with smoothing operators $R_j\in T^{-\infty}(\bfg_j,P_j)$. These identities yield $A^{-1}=B+R_0A^{-1}$ and $A^{-1}=B+A^{-1}R_1$. Thus we get 
 $$A^{-1}=B+R_0B+R_0P_0A^{-1}P_1R_1.$$
The right-hand side belongs to $T^0(\bfg^{(-1)},P_1,P_0)$. In fact, we can consider $T:=P_0A^{-1}P_1$ as a continuous map $H^0(g_1)\to H^0(g_0)$, hence  
$S:=R_0TR_1\in L^{-\infty}(\bfg^{(-1)})$ by assumption. Moreover $P_1S=SP_0=0$, showing 
$S\in T^{-\infty}(\bfg^{(-1)},P_1,P_0)$. 
\end{proof}

By abuse of notation we shall also write $A^{-1}:=B$ for $B$ from the previous theorem. This notation is reasonable, since $B:H^0(g_1,P_1)\to H^0(g_0,P_0)$ is the inverse of $A:H^0(g_0,P_0)\to H^0(g_1,P_1)$. 

\subsection{Reductions of orders}\label{sec:3.3}

In applications typically the fitration in \eqref{eq:pseudo} uses a parameter $\mu\in\gz$ or  $\mu\in\rz$ and the scale of Sobolev spaces \eqref{eq:sobolev} admits regularities $s\in\gz$ or $s\in\rz$. Of course, one is also interested in operators of order different from zero. A typical feature in pseudodifferential calculi is the existence of ``reductions of orders", which allows to restrict ones attention to the zero order case. In the present general setting this means $($to ask for$)$ the existence of operators $S_\mu\in L^\mu(\bfg)$ having an inverse 
$S_\mu^{-1}=S_{-\mu}\in L^{-\mu}(\bfg)$, for any $\mu$ and any admissible $\bfg=(g,g)$. 

In case of existence of such reductions of orders, $S_{\mu}^j\in L^\mu(\bfg_j)$, 
the study of $A\in T^\mu(\bfg,P_0,P_1)$ considered as an operator 
 $$A:H^s(g_0,P_0)\lra H^{s-\mu}(g_1,P_1)$$
is equivalent to the study of 
 $$\wt{A}:H^0(g_0,\wt{P}_0)\lra H^{0}(g_1,\wt{P}_1)$$
where 
 $$\wt{A}=S_{s-\mu}^1A S_{-s}^0\in L^0(\bfg,\wt{P}_0,\wt{P}_1)$$
with the two projections 
 $$\wt{P}_0=S_{s}^0P_0S_{-s}^0\in L^0(\bfg_0),\qquad 
   \wt{P}_1=S_{s-\mu}^1P_1S_{\mu-s}^1\in L^0(\bfg_1).$$
In fact, the following diagram is commutative: 
 $$
   \begin{CD}
     H^{s}(g_0,P_0) @> A >> H^{s-\mu}(g_1,P_1)\\
     @A{S_{-s}^0}AA @VV{S_{s-\mu}^1}V\\
     H^{0}(g_0,\wt{P}_0) @> \wt{A} >> H^{0}(g_1,\wt{P}_1)
   \end{CD}
 $$

\subsection{Ellipticity and principal symbols}\label{sec:3.4}

Above, ellipticity has been \emph{defined} as the existence of a parametrix, i.e., an inverse modulo smoothing remainders. In applications it is of course desirable to characterise ellipticity in other terms that are easier to verify. Typically, with a given operator $A$ one associates one or more ``$($homogeneous$)$ principal symbols", which can be thought of bundle morphisms between  finite or infinite dimensional vector bundles.\footnote{We shall identify operator-valued functions $\sigma:M\to\scrL(X,Y)$ with morphism acting between the trivial vector-bundles $M\times X$ and $M\times Y$.} Ellipticity is then aimed to be equivalent to the invertibility/bijectivity of the principal symbols. 

\begin{example}\label{ex:second}
With a classical pseudodifferential operator $A\in L^\mu_\cl(M;E_0,E_1)$, 
cf.\ Example \textrm{\ref{ex:first}}, we associate its homogeneous principal symbol, 
which is a vector bundle morphism
 $$\sigma^\mu(A):\;\pi^*E_0\lra \pi^*E_1,$$
where $\pi:S^*M\to M$ is the canonical projection of the unit co-sphere bundle $S^*M$ of the 
$($Riemannian$)$ manifold $M$ onto $M$ itself, and $\pi^*E_j$ denotes the pull-back of the vector 
bundle $E_j$. $A$ is elliptic if, and only if, $\sigma^\mu(A)$ is an isomorphism.  
\end{example}

In this section we assume that in $L^\star$ we have such a characterisation of ellipticity in 
terms of principal symbols and investigate how this structure descends to Toeplitz subalgebras. To 
this end let us call $L^\star$ a \emph{$\sigma$-algebra} if there exists a map 
 $$A\mapsto\sigma(A)=\big(\sigma_1(A),\ldots,\sigma_n(A)\big)$$
assigning to each $A\in L^0(\bfg)$ an $n$-tuple of bundle morphisms 
 $$\sigma_\ell(A):\;E_\ell(g_0)\lra E_\ell(g_1)$$
between $($finite or infinite dimensional$)$ Hilbert space bundles $E_\ell(g_j)$, such that the 
following is true:  
\begin{itemize}
 \item[$(1)$] The map respects the composition of operators, i.e., 
   $$\sigma(AB)=\sigma(A)\sigma(B)
     :=\big(\sigma_1(A)\sigma_1(B),\ldots,\sigma_n(A)\sigma_n(B)\big)$$
  whenever $A\in L^0(\bfg)$ and $B\in L^0(\bfg^\prime)$ as in \ref{eq:mult}.
 \item[$(2)$] $\sigma(R)=0$ for any smoothing operator $R$. 
 \item[$(3)$] $A$ is elliptic if, and only if, $\sigma(A)$ is invertible, i.e., 
  all $\sigma_\ell(A)$ are bundle isomorphisms. 
\end{itemize}
If additionally $L^\star$ is $*$-closed we also ask that  
\begin{itemize}
 \item[$(4)$] $\sigma(A^*)=\sigma(A)^*$, i.e., for any $\ell$, 
   $$\sigma_\ell(A^*)=\sigma_\ell(A)^*:\;E_1^\ell(g_1)\lra E_0^\ell(g_0),$$  
  where $\sigma_\ell(A)^*$ denotes the adjoint morphism $($obtained by taking fibrewise 
  the adjoint$)$. 
\end{itemize}

\begin{definition}
Let $L^\star$ be a $\sigma$-algebra and $A\in T^0(\bfg,P_0,P_1)$. 
Since the $P_j$ are projections also the associated bundle morphisms 
$\sigma_\ell(P_j)$ are projections in $E_\ell(g_j)$. Therefore its range  
 $$E_\ell(g_j,P_j):=\sigma_\ell(P_j)\big(E_\ell(g_j)\big)$$
is a subbundle of $E_\ell(g_j)$. We now define 
 $${\sigma}_\ell(A,P_0,P_1):E_\ell(g_0,P_0)\lra E_\ell(g_1,P_1)$$
by restriction of $\sigma_\ell(A)$ and then 
 $$\sigma(A,P_0,P_1)=\big(\sigma_1(A,P_0,P_1),\ldots,\sigma_n(A,P_0,P_1)\big).$$
\end{definition}

It is clear that if $A\in T^0(\bfg,P_0,P_1)$ is elliptic, then ${\sigma}(A,P_0,P_1)$ is invertible. In fact, if $B\in T^0(\bfg^{(-1)},P_1,P_0)$ is a parametrix to $A$ then ${\sigma}(B,P_1,P_0)$ is the inverse of ${\sigma}(A,P_0,P_1)$. 

\begin{theorem}\label{thm:sigma}
Let $L^\star$ be a $*$-closed $\sigma$-algebra. Then $T^\star$ is a $\sigma$-algebra. 
In particular, for $A\in T^0(\bfg,P_0,P_1)$ the following statements are equivalent: 
\begin{itemize}
 \item[$($a$)$] $A$ is elliptic. 
 \item[$($b$)$] ${\sigma}_\ell(A,P_0,P_1):E_\ell(g_0,P_0)\lra E_\ell(g_1,P_1)$ is an isomorphism 
  for $\ell=1,\ldots,n$.  
\end{itemize}
\end{theorem}
\begin{proof}
Let us show that $($b$)$ implies $($a$)$ $($the remaining statements are simple to see$)$. 
Let us define the operator 
 $$B:=A^*A+(1-P_0)^*(1-P_0)\in L^0(\bfg_0).$$
Applying the principal symbol map yields 
 $$\sigma_\ell(B)=\sigma_\ell(A)^*\sigma_\ell(A)+\sigma_\ell(1-P_0)^*\sigma_\ell(1-P_0).$$
Remark \ref{rem:opB} $($applied fibrewise$)$ shows that any $\sigma_\ell(B)$ is an isomorphism, 
hence $B$ is elliptic by assumption. Arguing as in the proof of Theorem \ref{thm:fred} we 
find a left-parametrix to $A$ and, by passing to adjoints, also a right-parametrix. 
Thus $A$ is elliptic. 
\end{proof}

The previous proof also shows that, under the assumptions of Theorem {\ref{thm:sigma}}, 
a left-parametrix $[$right-parametrix$]$ for $A$ exists, provided  
${\sigma}_\ell(A,P_0,P_1):E_\ell(g_0,P_0)\lra E_\ell(g_1,P_1)$ are (fibrewise) Fredholm 
monomorphisms $[$epimorphisms$]$ for all $\ell=1,\ldots,n$.  

\begin{example}\label{ex:third}
Let $P_j\in L^0_\cl(M;E_j,E_j)$ be two projections. Let us write 
 $$T^0_\cl(M;(E_0,P_0),(E_1,P_1))=P_1\, L^0_\cl(M;E_0,E_1)\,P_0$$
for the associated Toeplitz algbra and 
 $$H^s(M,E_j,P_j)=P_j\big(H^s(M,E_j)\big).$$ 
The principal symbol $\sigma(P_j):\,\pi^*E_j\to \pi^*E_j$ is a projection, and its range 
is a subbundle of $\pi^*E_j$ which we denote by $E_j(P_j)$. Then for 
$A\in T^0_\cl(M,(E_0,P_0),(E_1,P_1))$ the following statements are equivalent: 
\begin{itemize}
 \item[$($a$)$] $A$ is elliptic, i.e., has a parametrix $B\in T^0_\cl(M;(E_1,P_1),(E_0,P_0))$. 
 \item[$($b$)$] $A:H^0(M,E_0,P_0)\to H^0(M,E_1,P_1)$ is a Fredholm operator. 
 \item[$($c$)$] $\sigma(A):E_0(P_0)\to E_1(P_1)$ is an isomorphism. 
\end{itemize}
\end{example}

\begin{example}\label{ex:fourth}
Let $E_0,E_1$ be smooth vector bundles over $M$. By Swan's theorem $E_j$ is a subbundle of a 
trivial bundle; let us denote it by $\cz^{N_j}$. Then the projections $p_j:\cz^{N_j}\to E_j$ 
can be considered as zero-order pseudodifferential projections 
$P_j\in L^0_\cl(M;\cz^{N_j},\cz^{N_j})$. Then we have an identification of $L^\mu_\cl(M;E_0,E_1)$ 
with the Toeplitz subalgebra $P_0\,L^\mu_\cl(M;\cz^{N_0},\cz^{N_1})\,P_1$. 
\end{example}

Similarly as discussed in Section \ref{sec:3.3}, the existence of reductions of orders allows a straightforward extension of the above result from zero order operators to operators of general order. This fact we shall use below in our examples without further commenting on it. 

\subsection{Parameter-dependent operators}\label{sec:3.5}

For the analysis of resolvents of differential operators, calculi of \emph{parameter-dependent} pseudodifferential operators can be used very effectively. With easy modifications, the above abstract approach can also capture features of parameter-dependent calculi. We shall give some details in this subsection. 

In the following let $\Lambda$ coincide with $\rz^\ell$ or be a sectorial domain in the complex plane. We now assume that the elements of $L^*$ are not single operators, but families/functions of operators $\lambda\mapsto A(\lambda)$. To make clear that we deal with families of operators we shall use notations like $L^*(\Lambda)$, $L^\mu(\bfg;\Lambda)$ and denote elements by $A(\lambda)$, $B(\lambda)$, etc. 

With the notation from Section \ref{sec:2} we shall assume that 
$A(\lambda)\in L^\mu(\bfg;\Lambda)$ induces continuous oprators
 $$A(\lambda):\;H^s(g_0)\lra H^{s-\mu}(g_1)$$ 
for all $\lambda$ and all $s$. 
Smoothing operators $R(\lambda)\in L^{-\infty}(\bfg;\Lambda)$ are required to satisfy  
 $$\|R(\lambda)\|_{H^s(g_0),H^{t}(g_1)}\xrightarrow{|\lambda|\to\infty}0$$
for any $s,t$, where $\|\cdot\|_{X,Y}$ denotes the operator norm for operators $X\to Y$. 

\begin{definition}
$A(\lambda)\in L^0(\bfg;\Lambda)$ is called \emph{elliptic} $($or \emph{parameter-elliptic}$)$  
if there exists a $B(\lambda)\in L^{0}(\bfg^{(-1)};\Lambda)$ such that 
\begin{align*}
 R_0(\lambda)&:=B(\lambda)A(\lambda)-1\in L^{-\infty}(\bfg_0;\Lambda),\\ 
 R_1(\lambda)&:=A(\lambda)B(\lambda)-1\in L^{-\infty}(\bfg_1;\Lambda).
\end{align*}
Any such operator $B(\lambda)$ is called a parametrix of $A(\lambda)$. 
\end{definition}

An elliptic $A(\lambda)$ induces isomorphisms $H^0(g_0)\to H^{0}(g_1)$ for sufficiently 
large $\lambda$, since $1+R_j(\lambda)$ is invertible due to the decay property of smoothing 
remainders. If we assume that there exist smoothing $S_j(\lambda)$ 
such that $(1+R_j(\lambda))^{-1}=1+S_j(\lambda)$ for sufficiently large $\lambda$, 
we can conclude that there exists a parametrix $B(\lambda)$ that equals 
$A(\lambda)^{-1}$ for large enough $\lambda$. 

\begin{definition}
We call $L^\star(\Lambda)$ \emph{inverse-closed}, if for any 
$R(\lambda)\in L^{-\infty}(\bfg;\Lambda)$ with admissible weight $\bfg=(g,g)$ 
there exists an $S(\lambda)\in L^{-\infty}(\bfg;\Lambda)$ such that 
 $$(1+R(\lambda))(1+S(\lambda))=(1+S(\lambda))(1+R(\lambda))=1$$ 
for sufficiently large $\lambda$. 
\end{definition}

Using the notation of the previous definition, we have 
 $$\big(1+R(\lambda)\big)^{-1}=1-R(\lambda)
    +R(\lambda)\big(1+R(\lambda)\big)^{-1}R(\lambda)$$
whenever the inverse exists. 
Hence we see that $L^\star(\Lambda)$ is inverse closed if, and only if, for any such 
$R(\lambda)$ there exists an $R^\prime(\lambda)\in L^{-\infty}(\bfg;\Lambda)$ such 
that 
 $$R^\prime(\lambda)=R(\lambda)\big(1+R(\lambda)\big)^{-1}R(\lambda)$$
for sufficiently large $\lambda$. 

\begin{example}
With the notation introduced in Example {\rm\ref{ex:first}}, let 
 $$L^\mu(\bfg;\Lambda)=L^\mu_\cl(M;E_0,E_1;\Lambda)$$
be the space of classical parameter-dependent pseudodifferential operators of order 
$\mu$.\footnote{i.e., the local symbols satisfy uniform estimates of the form 
 $$|D^\alpha_\xi D^\beta_x D^\gamma_\lambda a(x,\xi,\lambda)|\le C 
(1+|\xi|+|\lambda|)^{\mu-|\alpha|-|\gamma|},$$ 
and have expansions into components homogeneous in $(\xi,\lambda)$.} 
The smoothing operators are rapidly decreasing in $\lambda$ with values in the smoothing 
operators on $M$, 
 $$L^{-\infty}(M;E_0,E_1;\Lambda)=\scrS(\Lambda,X),\qquad 
     X=L^{-\infty}(M;E_0,E_1)\cong \scrC^\infty(M,E_0\boxtimes E_1).$$
Now it is straightforward to verify the inverse closedness, using the fact that for a rapidly 
decreasing function $r(\lambda)$ also 
 $$r^\prime(\lambda):=\chi(\lambda)r(\lambda)\big(1+r(\lambda)\big)^{-1}r(\lambda)$$
is reapidly decreasing, where $\chi$ is a zero excision function that vanishes where the 
inverse does not exist. 
\end{example}

As before we can now consider Toeplitz algebras 
 $$T^\mu(\bfg,P_0,P_1;\Lambda)=P_1(\lambda)L^\mu(\bfg;\Lambda)P_0(\lambda)$$
with projections $P_j(\lambda)\in L^\mu(\bfg_j;\Lambda)$. An element 
$A(\lambda)\in T^0(\bfg,P_0,P_1;\Lambda)$ is called elliptic if there exists a
$B(\lambda)\in T^0(\bfg,P_1,P_0;\Lambda)$ such that 
\begin{align*}
 B(\lambda)A(\lambda)-P_0(\lambda)&\in T^{-\infty}(\bfg_0,P_0,P_0;\Lambda),\\ 
 A(\lambda)B(\lambda)-P_1(\lambda)&\in T^{-\infty}(\bfg_1,P_1,P_1;\Lambda).
\end{align*}
Inverse-closedness of the Toeplitz algebras now means that to any projection 
$P(\lambda)\in L^\mu(\bfg;\Lambda)$, $\bfg=(g,g)$, and any 
$R(\lambda)\in T^{-\infty}(\bfg,P,P;\Lambda)$ there exists an 
$S(\lambda)\in T^{-\infty}(\bfg,P,P;\Lambda)$ such that 
 $$(P(\lambda)+R(\lambda))(P(\lambda)+S(\lambda))
     =(P(\lambda)+S(\lambda))(P(\lambda)+R(\lambda))=P(\lambda)$$ 
for sufficiently large $\lambda$. In this case, to any elliptic $A(\lambda)$ there always exists 
a parametrix $B(\lambda)$ such that  
 $$B(\lambda)A(\lambda)=P_0(\lambda),\qquad  
   A(\lambda)B(\lambda)=P_1(\lambda)$$
for large $\lambda$. It is not difficult to see that inverse-closedness of $L^\star(\Lambda)$ implies that of 
$T^\star(\Lambda)$. 

Assuming that ellipticity in $L^\star(\Lambda)$ is characterized by the invertibility of certain principal 
symbols, i.e., $L^\star(\Lambda)$ is a $\sigma$-algebra, we can now show as in Section \ref{sec:3.4}: 

\begin{theorem}
If $L^\star(\Lambda)$ is an inverse-closed, $*$-closed $\sigma$-algebra, then 
$T^\star(\Lambda)$ is an inverse-closed $\sigma$-algebra.
\end{theorem}

\section{$SG$-pseudodifferential operators}\label{sec:4}

If $X$ is a Fr\'echet space let us denote by $S^\mu(\rz^m,X)$ the Fr\'echet space of all smooth 
functions $a:\rz^m\to X$ satisfying estimates 
 $$\|D^\gamma_z a(z)\|\le C_\gamma (1+|z|)^{\mu-|\gamma|}$$
uniformly in $z\in\rz^m$ for any multi-index $\gamma$ and any semi-norm $\|\cdot\|$ of $X$ 
$($the constant $C_\gamma$ depends also on the semi-norm$)$. With $S^{(\mu)}(\rz^m,X)$ we 
denote the space of all smooth functions $a:\rz^m\setminus\{0\}\to X$ of the form 
 $$a(z)=|z|^\mu \wt{a}\big({z}/{|z|}\big),\qquad \wt{a}:S^{m-1}\to X, $$
where $\sz^{m-1}$ denotes the unit-sphere in $\rz^m$. Moreover, $S^{\mu}_\cl(\rz^m,X)$ 
denotes the subspace of symbols $a\in S^\mu(\rz^m,X)$ that have asymptotic expansions 
into homogeneous components: There exist $a^{(\mu-j)}\in S^{(\mu-j)}(\rz^m,X)$ such that 
 $$a-\sum_{j=0}^{N-1} \chi a^{(\mu-j)}\;\in\; S^{\mu-N}(\rz^m,X)$$
for any positive integer $N$ and with $\chi(z)$ being a zero-excision 
function $($i.e., $\chi:\rz^m\to\rz$ smooth, vanishing in an open neighborhood of 
$z=0$, and being constant $1$ outside some compact set$)$. The function $a^{(\mu)}$ is called 
the principal component of $a$. 

The class of pseudodifferential symbols we now consider are, roughly speaking, classical 
both in the $x$-variable and the corresponding co-variable $\xi$ $($for precise details 
we refer the reader to \cite{EgSc}$)$.

\begin{definition}
For $\mu,m\in\rz$ and $N_0,N_1\in\nz$ let us define 
 $$S^{\mu,m}_\cl(\rz^n\times\rz^n,N_0,N_1):=
   S^{m}_\cl\big(\rz^n_x,S^{\mu}_\cl(\rz^n_\xi,\scrL(\cz^{N_0},\cz^{N_1}))\big).$$
The space of associated pseudodifferential operators, 
 $$(Au)(x)=[\op(a)u](x)=\int e^{ix\xi}a(x,\xi)\,\wh{u}(\xi)\,\dbar\xi,$$
we shall denote by $L^{\mu,m}_\cl(\rz^n,N_0,N_1)$. 
\end{definition}

The class of regularizing operators, 
 $$L^{-\infty,-\infty}(\rz^n,N_0,N_1):=
   \mathop{\mbox{\Large$\cap$}}_{\mu,m\in\rz}L^{\mu,m}_\cl(\rz^n,N_0,N_1),$$
consists of all integral operators having a kernel $k(x,y)\in \scrS(\rz^{2n}_{(x,y)})$, 
the space of rapidly decreasing functions $($with respect to the standard Lebesgue 
measure on $\rz^n)$. 

The natural scale of Sobolev spaces such operators act in is given by 
 $$H^{s,\delta}(\rz^n,N):=(1+|x|^2)^{\delta/2}H^s(\rz^n,\cz^N),\qquad s,\delta\in\rz,$$
i.e., the standard $\cz^N$-valued Sobolev spaces on $\rz^n$ multiplied by a weight function. 
Then $A\in L^{\mu,m}_\cl(\rz^n,N_0,N_1)$ induces continuous operators 
 $$A:\;H^{s,\delta}(\rz^n,N_0)\lra H^{s-\mu,\delta-m}(\rz^n,N_1)$$
for any choice of $s$ and $\delta$.  

By passing to the principal component with respect to $x$, or with respect to $\xi$, or 
simultaneously with respect to both $x$ and $\xi$ we associate with $A=\op(a)$ three 
principal symbols, which are bundle morphisms 
\begin{align*}
 \sigma^\mu(A)&:\;(\rz^n_x\times\sz^{n-1}_\xi)\times\cz^{N_0}\lra
   (\rz^n_x\times\sz^{n-1}_\xi)\times\cz^{N_1},\\  
 \sigma_m(A)&:\;(\sz^{n-1}_x\times\rz^{n}_\xi)\times\cz^{N_0}\lra
   (\sz^{n-1}_x\times\rz^{n}_\xi)\times\cz^{N_1},\\  
 \sigma^\mu_m(A)&:\;(\sz^{n-1}_x\times\sz^{n-1}_\xi)\times\cz^{N_0}\lra
   (\sz^{n-1}_x\times\sz^{n-1}_\xi)\times\cz^{N_1}.  
\end{align*}
Let us then set 
 $$\sigma(A)=\Big(\sigma^\mu(A),\sigma_m(A),\sigma^\mu_m(A)\Big),\qquad 
   A\in L^{\mu,m}_\cl(\rz^n,N_0,N_1).$$
Ellipticity of $A$ is defined as the invertibility of $($all three components of$)$ $\sigma(A)$. 
It is well-known that ellipticity is equivalent to the existence of a parametrix 
$B\in L^{-\mu,-m}_\cl(\rz^n,N_1,N_0)$ modulo remainders in $L^{-\infty,-\infty}$, and it is 
equivalent to $A:H^{s,\delta}(\rz^n,N_0)\to H^{s-\mu,\delta-m}(\rz^n,N_1)$ being a Fredholm 
operator for some $($and then for all$)$ $s,\delta\in\rz$. For the latter result see \cite{Grie}. 
 
With two projections $P_j\in L^{0,0}_\cl(\rz^n,N_j,N_j)$, $j=0,1$, let us write 
 $$T^{\mu,m}_\cl(\rz^n,(N_0,P_0),(N_1,P_1))
   =P_1\,L^{\mu,m}_\cl(\rz^n,N_0,N_1)\,P_0$$
for the associated Toeplitz subalgebra and 
 $$H^{s,\delta}(\rz^n,N_j,P_j)=P_j\big(H^{s,\delta}(\rz^n,N_j)\big)$$
for the associated scale of projected Sobolev spaces. 
The principal symbols $\sigma^0(P_j)$, $\sigma_0(P_j)$ and $\sigma_0^0(P_j)$ of $P_j$ 
are itself projections, thus their ranges define subbundles 
\begin{align*}
 E^0(N_j,P_j)&\subset
   (\rz^n_x\times\sz^{n-1}_\xi)\times\cz^{N_j},\\  
 E_0(N_j,P_j)&\subset
   (\sz^{n-1}_x\times\rz^{n}_\xi)\times\cz^{N_j},\\  
 E^0_0(N_j,P_j)&\subset
   (\sz^{n-1}_x\times\sz^{n-1}_\xi)\times\cz^{N_j}.  
\end{align*}
The principal symbol ${\sigma}(A,P_0,P_1)$ consists of the three components 
\begin{align*}\label{eq:sg-symb}
\begin{split}
 {\sigma}^\mu(A,P_0,P_1)&:\,E^0(N_0,P_0)\lra E^0(N_1,P_1),\\
 {\sigma}_m(A,P_0,P_1)&:\,E_0(N_0,P_0)\lra E_0(N_1,P_1),\\
 {\sigma}^\mu_m(A,P_0,P_1)&:\,E^0_0(N_0,P_0)\lra E^0_0(N_1,P_1).
\end{split}
\end{align*}
obtained by the restriction of the corresponding symbols of $A$.

\begin{theorem}
For $A\in T^{\mu,m}_\cl(\rz^n,(N_0,P_0),(N_1,P_1))$ the following statements are 
equivalent: 
\begin{itemize}
 \item[$($a$)$] $A$ is elliptic, i.e., has a parametrix 
  $B\in T^{-\mu,-m}_\cl(M,(N_1,P_1),(N_0,P_0))$ $($modulo remainders in $T^{-\infty,-\infty})$. 
 \item[$($b$)$] $A:H^{s,\delta}(\rz^n,N_0,P_0)\to H^{s-\mu,\delta-m}(\rz^n,N_1,P_1)$ is a 
  Fredholm operator for some $s,\delta\in\rz$. 
 \item[$($c$)$] The morphisms ${\sigma}^\mu(A,P_0,P_1)$, ${\sigma}_m(A,P_0,P_1)$, and 
  ${\sigma}^\mu_m(A,P_0,P_1)$ are isomorphisms. 
\end{itemize}
In this case, $($b$)$ is true for arbitrary $s,\delta\in\rz$.  
\end{theorem}

In case $\mu=m=0$ and $s=\delta=0$ this theorem is just a particular case of Theorems  
\ref{thm:fred} and \ref{thm:sigma}, with 
\begin{align*}
 L^{\mu}(\bfg)&=L^{\mu,\mu}_\cl(\rz^n,N_0,N_1),\qquad \bfg=\big((\rz^n,N_0),(\rz^n,N_1)\big)\\
 H^s(g)&=H^{s,s}(\rz^n,N),\qquad g=(\rz^n,N).
\end{align*}
The general case is obtained by the use of order reductions, analogously as described in 
Section 3.3 $($with the minor modification that we have here two parameters $\mu$ and $m)$. 
In fact, the operators having symbol $[x]^m[\xi]^\mu I_N$, where $I_N$ is the $N\times N$-unit 
matrix and $[\cdot]:\rz^n\to(0,\infty)$ is a smooth function that coincides with 
$|\cdot|$ outside the unit-ball, induce isomorphisms 
$H^{s,\delta}(\rz^n,N)\lra H^{s-\mu,\delta-m}(\rz^n,N)$
for arbitrary $s$ and $\delta$. Theorem \ref{thm:spectral} implies the spectral invariance: 

\begin{theorem}
If $A\in T^{\mu,m}_\cl(\rz^n,(N_0,P_0),(N_1,P_1))$ induces an isomorphism 
 $$H^{s,\delta}(\rz^n,N_0,P_0)\lra H^{s-\mu,\delta-m}(\rz^n,N_1,P_1)$$ 
for some $s,\delta\in\rz$, 
then this is true for any $s,\delta\in\rz$ and the inverse $A^{-1}$ belongs to 
$T^{-\mu,-m}_\cl(M,(N_1,P_1),(N_0,P_0))$. 
\end{theorem}

\section{Operators on manifolds with conic singularities}\label{sec:5}

We are now going to discuss the cone algebra of Schulze.
For detailed presentations of the cone algebra we refer to \cite{EgSc} and to \cite{Seil}. To keep the presentation lean we shall focus on a version of the cone algebra which is sufficient for the characterization of the Fredholm property in certain weighted Sobolev spaces. 

Let $\bz$ be a smooth compact manifold with boundary $X:=\partial\bz$ with $\mathrm{dim}\,X=n$. We identify a collar neighborhood of $X$ with $[0,1)\times X$ and fix corresponding variables $(t,x)$ near the boundary. The typical differential operators we consider on $($the interior$)$ of $\bz$ are away from the boundary usual differential operators with smooth coefficients, while near the boundary they can be written in the form 
\begin{equation}\label{eq:conediff1}
 A=t^{-\mu}\sum_{j=0}^\mu a_j(t)(-t\partial_t)^j,\qquad 
 a_j\in\scrC^\infty([0,1),\mathrm{Diff}^{\mu-j}(X)),
\end{equation}
with coefficients taking values in the space of differential operators on $X$. 
The Laplacian with respect to a Riemannian 
metric which near the boundary has the form $dt^2+t^2dx^2$ is of that form, with $\mu=2$. 
Such "cone differential operators" act in a scale of weighted Sobolev spaces 
 $$\scrH^{s,\gamma}(\bz)=k^\gamma \scrH^{s,0}(\bz),\qquad s,\gamma\in\rz,$$
where $k$ is a smooth positive function on the interior of $\bz$ that coincides with $k(t,x)=t$ 
near the boundary, while $u\in\scrH^{s,0}(\bz)$ for $s\in\nz_0$ if and only if 
$u\in H^s_{\mathrm{loc}}(\mathrm{int}\,\bz)$ and 
 $$t^{\frac{n+1}{2}-\gamma}(t\partial_t)^jD^\alpha_x u(t,x)\in 
   L^2\Big((0,1)\times X,\frac{dt}{t}dx\Big)$$
for all $s+|\alpha|\le s$; this definition can be extended to real $s$ by interpolation and 
duality. A differential operator as above induces then maps 
\begin{equation}\label{eq:conediff2}
 A:\;\scrH^{s,\gamma}(\bz)\lra \scrH^{s-\mu,\gamma-\mu}(\bz).
\end{equation}
If we set $\Gamma_\beta=\{z\in\cz\st \re z=\beta\}$, $\beta\in\rz$, and let 
\begin{align*}
 (M_\gamma u)(z)&=\int_0^\infty t^{z}u(t)\,\frac{dt}{t},\qquad z\in\Gamma_{1/2-\gamma},\\
 (M_\gamma^{-1} v)(t)&=\frac{1}{2\pi i}\int_{\Gamma_{1/2-\gamma}} t^{-z}v(z)\,dz,\qquad t>0,
\end{align*}
denote the weighted Mellin transform and its inverse, then we can write \eqref{eq:conediff2} 
with $A$ from \eqref{eq:conediff1} and $u$ supported close to the boundary as a Mellin 
pseudodifferential operator, 
 $$(Au)(t)=t^{-\mu}\big(\op_M^{\gamma-n/2}(h)u\big)(t)
   =t^{-\mu}M_{\gamma-n/2}^{-1}\Big(h(t,z)(M_{\gamma-n/2} u)(z)\Big)$$
with $h(t,z)=\sum_{j=0}^\mu a_j(t)z^j$. 

\begin{definition}\label{def:calg}
The cone algebra $C^{\mu-j}(\bz,\gamma,\gamma-\mu)$, $\mu\in\gz$, $j\in\nz_0$, consists of all 
operators of the form 
 $$A=\omega\,t^{-\mu}\op_M^{\gamma-n/2}(h)\,\omega_0+(1-\omega)\,\wt{A}\,(1-\omega_1)+G,$$
where $\omega,\omega_0,\omega_1$ are smooth functions supported in $[0,1)\times X$ which are 
equal to 1 near the boundary and satisfy $\omega\omega_0=\omega$, $\omega\omega_1=\omega_1$ and 
 $$h(t,z)=t^jh_0(t,z)+\delta_{j0}h_{-\infty}(z),$$
$(\delta_{j0}=1$ if $j=0$ and $\delta_{j0}=0$ if $j\ge1)$ where
\begin{itemize}
 \item[$($i$)$] $h_0$ is smooth in $t\in[0,1)$ and entire in $z$ with values in 
  $L^{\mu-j}_\cl(X)$ and 
   $$h_0(t,\beta+i\tau)\in\scrC^\infty\big([0,1),L^{\mu-j}_\cl(X;\rz_\tau)\big)$$
  uniformly for $\beta$ in compact intervals,\footnote{$L^{\nu}_\cl(X;\rz_\tau)$ denotes the 
  class of parameter-dependent pseudodifferential operators of order $\nu$, 
  with parameter $\tau\in\rz$}  
 \item[$($ii$)$] $h_{-\infty}$ is holomorphic with values in $L^{-\infty}(X)$ in a strip 
  $\{z\st |(n+1)/2-\gamma-\re z|<\eps\}$ with some $\eps>0$, and  
  $h_{-\infty}(\beta+i\tau)\in L^{-\infty}_\cl(X;\rz_\tau)$ uniformly in $\beta$, 
 \item[$($iii$)$] $\wt{A}\in L^{\mu-j}_\cl(\mathrm{int}\,\bz)$ is a usual pseudodifferential 
  operator,
 \item[$($iv$)$] $G$ is an integral operator with respect to the measure $t^n{dt}{dx}$ 
  with kernel in 
  $\scrH^{\infty,\gamma-\mu+\eps}(\bz)\wh{\otimes}_\pi \scrH^{\infty,-\gamma+\eps}(\bz)$ 
  for some $\eps>0$.\footnote{$\wh{\otimes}_\pi$ denotes the completed projective tensor product.} 
\end{itemize} 
The operators $G$ of $($iv$)$ constitute the class of smoothing operators, denoted by 
$C^{-\infty}(\bz,\gamma,\gamma-\mu)$. 
\end{definition}

It is straight-forward to extend the previous definition to operators acting between sections 
into vector bundles $E_0$ and $E_1$ over $\bz$, yielding algebras 
\begin{equation}\label{eq:calg}
 C^{\mu-j}(\bz,(\gamma,E_0),(\gamma-\mu,E_1)).
\end{equation}
With any element $A$ of \eqref{eq:calg} with $j=0$ we associate two principal symbols. 
The first is 
\begin{equation}\label{eq:csymb1}
 \sigma_c^\mu(A):\;\pi^*_cE_0\lra \pi^*_cE_1,
\end{equation}
where $\pi_c:T^*_c\bz\setminus0\to\bz$ denotes the canonical projection of the ``compressed" co-tangent bundle of $\bz$ onto $\bz$; for a precise definition see \cite{GKM}. Roughly speaking, over the interior of $\bz$ this symbol recovers the usual principal symbol, while for $t\to0$ the product $t\tau$ $($arising from the totally characteristic derivative $t\partial_t)$ is replaced by a single variable $\wt\tau$. The second is the so-called \emph{conormal symbol}. Using the notation from Definition \ref{def:calg} it is the function 
 $$h(0,z):\;H^s(X)\lra H^{s-\mu}(X),\qquad \re z=\frac{n+1}{2}-\gamma,$$
where the choice of $s$ does not play a role. Identifying operator-valued functions with morphisms 
in trivial bundles, for an $A$ from \eqref{eq:calg} with $j=0$ we get 
\begin{equation}\label{eq:csymb2}
 \sigma_M^\mu(A):\;\Gamma_{\frac{n+1}{2}-\gamma}\times H^s(X,E_0^\prime)\lra 
 \Gamma_{\frac{n+1}{2}-\gamma}\times H^{s-\mu}(X,E_1^\prime),
\end{equation}
where $E_j^\prime$ denotes the restriction of $E_j$ to $X=\partial\bz$. 

For convenience of notation let us now set $\gamma_0=\gamma$ and $\gamma_1=\gamma-\mu$. Given two 
projections $P_j\in C^0(\bz,(\gamma_j,E_j),(\gamma_j,E_j))$, the associated Toeplitz subalgebras 
are 
 $$T^{\mu-j}(\bz,(\gamma_0,E_0,P_0),(\gamma_1,E_1,P_1))=
   P_1\,C^{\mu-j}(\bz,(\gamma_0,E_0),(\gamma_1,E_1))\,P_0.$$
The scales of projected Sobolev spaces are 
 $$\scrH^{s,\gamma_j}(\bz,E_j,P_j)=P_j\big(\scrH^{s,\gamma_j}(\bz,E_j)\big).$$
Since both $\sigma_c^0(P_j)$ and $\sigma_M^0(P_j)$ are projections we obtain subbundles 
\begin{align*}
 E^c_j(P_j)&:=\sigma_c^0(P_j)\big(\pi^*_cE_j\big)\subset \pi^*_cE_j,\\
 E^M_j(P_j)&:=\sigma_M^0(P_j)\big(\Gamma_{\frac{n+1}{2}-\gamma_j}\times H^s(X,E_j^\prime)\big)
  \subset \Gamma_{\frac{n+1}{2}-\gamma_j}\times H^s(X,E_j^\prime).
\end{align*}
The principal symbol ${\sigma}(A,P_0,P_1)$ consists of the two components 
\begin{align*}\label{eq:c-symb}
 {\sigma}^\mu_c(A,P_0,P_1):\,E^c_0(P_0)\lra E^c_1(P_1),\qquad 
 {\sigma}^\mu_M(A,P_0,P_1):\,E_0^M(P_0)\lra E_1^M(P_1),
\end{align*}
induced by the restriction of $\sigma_c^\mu(A)$ and $\sigma_M^\mu(A)$, respectively. 

\begin{theorem}
For $A\in T^{\mu}(\bz,(\gamma,E_0,P_0),(\gamma-\mu,E_1,P_1))$ the following statements are 
equivalent: 
\begin{itemize}
 \item[$($a$)$] $A$ is elliptic, i.e., has a parametrix 
  $B\in T^{-\mu}(\bz,(\gamma-\mu,E_1,P_1),(\gamma,E_0,P_0))$.   
 \item[$($b$)$] $A:\scrH^{s,\gamma}(\bz,E_0,P_0)\to \scrH^{s-\mu,\gamma-\mu}(\bz,E_1,P_1)$ is a 
  Fredholm operator for some $s\in\rz$. 
 \item[$($c$)$] Both morphisms ${\sigma}^\mu_c(A,P_0,P_1)$ and ${\sigma}^\mu_M(A,P_0,P_1)$ are 
  isomorphisms. 
\end{itemize}
In this case, $($b$)$ is true for arbitrary $s\in\rz$.  
\end{theorem}

In fact, by the existence of suitable reductions of orders, we can reduce the proof of the 
previous theorem to the case $\mu=s=\gamma=0$. This is then a particular case of Theorems  
\ref{thm:fred} and \ref{thm:sigma}, with 
\begin{align*}
 L^{\mu}(\bfg)&=C^{\mu}(\bz,(0,E_0),(0,E_1)),\qquad 
 \bfg=\big((\bz,E_0),(\bz,E_1)\big)\\
 H^s(g)&=\scrH^{s,0}(\bz,E),\qquad g=(\bz,E).
\end{align*}
The Toeplitz algebras are spectrally invariant, as a consequence of Theorem 
\ref{thm:spectral}.

\begin{theorem}
If $A\in T^{\mu}(\bz,(\gamma,E_0,P_0),(\gamma-\mu,E_1,P_1))$ induces an isomorphism 
$\scrH^{s,\gamma}(\bz,E_0,P_0)\to \scrH^{s-\mu,\gamma-\mu}(\bz,E_1,P_1)$ for some $s$, 
then for all $s$ and 
  $$A^{-1}\in T^{-\mu}(\bz,(\gamma-\mu,E_1,P_1),(\gamma,E_0,P_0)).$$
\end{theorem}

\section{Boundary value problems}\label{sec:6}

We now consider Boutet de Monvel's algebra for boundary value problems and apply the 
general result to show the invertibilty of the Stokes operator within a certain Toeplitz 
subalgebra.  

\subsection{Operators on the half space}\label{sec:6.1}
We give a short presentation of Boutet de Monvel's algebra on the half-space 
$M:=\overline{\rz}_+^n=\rz^{n-1}\times[0,\infty)$. We shall employ the 
splitting of variables $x=(x^\prime,x_n)$ and $\xi=(\xi^\prime,\xi_n)$. 

The space $\calB^{\mu,d}(M;(n_0,j_0),(n_1,j_1))$ with $\mu\in\gz$ $($the \emph{order}$)$, 
$d\in\nz_0$ $($the \emph{type}$)$, and $n_\ell,j_\ell\in\nz_0$ consists of all operators of 
the form $\calA=\begin{pmatrix}A_++G&K\\ T & Q \end{pmatrix}$ where the single entries are 
described in the following:\footnote{
Also the cases $j_0=0$ or $j_1=0$ are allowed. If $j_0=0$ the operators have the form 
$\calA=\begin{pmatrix}A_++G\\ T\end{pmatrix}$, if $j_1=0$ the form 
$\calA=\begin{pmatrix}A_++G&K\end{pmatrix}$, and if $j_0=j_1=0$ the form 
$\calA=A_++G$.} 
\begin{itemize}
 \item[$(1)$] Denoting by $e_+$ the operator of extension by zero from $M$ to $\rz^n$ and by 
  $r_+$ the operator of restriction from $\rz^n$ to $M$, $A_+=r_+Ae_+$ with a classical 
  pseudodifferential operator $A=\op(a)$ of order $\mu$ whose $(n_1\times n_0)$-matrix valued 
  symbol satisfies the transmission condition at $x_n=0$, i.e., 
  if $a\sim\sum_j a_{(\mu-j)}$ is the expansion into homogeneous components, then 
   \begin{align*}
     D^k_{x_{n}}D^\alpha_{(\xi,\tau)} & a_{(\mu-j)}(x^\prime,0,0,1,0)\\
     &=(-1)^{\mu-j-|\alpha|} D^k_{x_{n}}D^\alpha_{(\xi,\tau)}a_{(\mu-j)}(x^\prime,0,0,-1,0) 
   \end{align*}   
   for all $j,k$ and $\alpha$. 
 \item[$(2)$] $T$ is a \emph{trace operator} of order $\mu$ and type $d$, i.e., 
  $T = \sum\limits_{j=0}^{d-1} S_{j} \gamma_{j} + T_0$,  
  where $\gamma_{j}u=\frac{\partial^j u}{\partial x_n^j}\Big|_{x_n=0}$, 
  $S_{j}$ is a pseudodifferential operator  of order $\mu-j$ on $\rz^{n-1}$ with 
  $(j_1\times n_0)$-matrix valued symbol, and $T_0$ is of the form 
   $$(T_0u)(x^\prime) = \int_{\rz^{n-1}}\int_{0}^\infty 
       e^{i x^\prime \xi^\prime}t(x^\prime,\xi^\prime;[\xi^\prime]x_n) \,
       (\calF_{x^\prime\to\xi^\prime}u)(\xi^\prime, x_{n})  \, dx_{n} \dbar\xi^\prime, $$ 
  where $t(x^\prime,\xi^\prime;s)$ is smooth and rapidly decreasing as a function of 
  $s\in[0,\infty)$, 
  is a $(j_1\times n_0)$-matrix valued symbol of order $\mu+1/2$ with respect to 
  $(x^\prime,\xi^\prime)$,  
  and $[\,\cdot\,]$ denotes a smooth positive function with $[\,\cdot\,]=|\cdot|$ outside a 
  neighborhood of the origin. 
 \item[$(3)$] A \emph{Poisson operator} $K$ of order $\mu$ $(d$ is  not relevant here$)$, i.e., 
    $$(Kv)(x)=\int_{\rz^{n-1}} e^{i x^\prime \xi^\prime}
        k(x^\prime,\xi^\prime;[\xi^\prime]x_n)\hat{v}(\xi^{\prime})\dbar \xi^\prime, $$
   where $k(x^\prime,\xi^\prime;s)$ has a structure analogous to $t$ from $(2)$, but $k$ is a 
   symbol of order $\mu-1/2$ with respect to $(x^\prime,\xi^\prime)$ and is 
   $(n_1\times j_0)$-matrix valued. 
 \item[$(4)$] A \emph{singular Green operator} of order $\mu$ and type $d$, i.e.,   
  $G=\sum\limits_{j=0}^{d-1}K_{j}\gamma_{j} + G_0$ with Poisson operators $K_{j}$ of 
  order $\mu-j$, and $G_0$ of the form 
    $$(G_0 u)(x) = \int_{\rz^{n-1}}\int_{0}^\infty e^{ix^\prime\xi^\prime}
        g(x^\prime,\xi^\prime;[\xi^\prime]x_n,[\xi^\prime]y_n)(\calF_{x^\prime\to\xi^\prime}u)
        (\xi^\prime, y_{n}) dy_{n} \dbar \xi^\prime,$$ 
  where $g(x^\prime,\xi^\prime;s,t)$ is smooth and rapidly decreasing as a function of 
  $(s,t)\in[0,\infty)^2$, and 
  is a $(n_1\times n_0)$-matrix valued symbol of order $\mu+1$ with respect to 
  $(x^\prime,\xi^\prime)$. 
 \item[$(5)$] A usual pseudodifferential operator $Q$ of order $\mu$ on $\rz^{n-1}$ with 
  $(j_1\times j_0)$-matrix valued symbol. 
\end{itemize}
Any such $\calA$ induces continuous operators 
\begin{equation}\label{eq:map}
 \calA:\;\begin{matrix} H^s(M,\cz^{n_0})\\ \oplus\\  H^s(\partial M,\cz^{j_0})\end{matrix}
  \longrightarrow
  \begin{matrix}H^{s-\mu}(M,\cz^{n_1})\\ \oplus\\  H^{s-\mu}(\partial M,\cz^{j_1})\end{matrix},
  \qquad s>d-\frac{1}{2}.
\end{equation}
With $\calA$ one associates two different principal symbols. The first is the usual homogeneous 
principal symbol of $A$ $($considered on $M$ including its boundary$)$, which we denote now with 
$\sigma_\psi^\mu(\calA)$. The second is the principal \emph{boundary symbol} 
 $$\sigma^\mu_\partial(\calA)(x^\prime,\xi^\prime):\;
   \begin{matrix} H^s(\rz_+,\cz^{n_0})\\ \oplus\\  \cz^{j_0}\end{matrix}
   \longrightarrow
   \begin{matrix}H^{s-\mu}(\rz_+,\cz^{n_1})\\ \oplus\\ \cz^{j_1}\end{matrix}$$
$($the choice of $s>d-\frac{1}{2}$ is arbitrary$)$ defined on $S^*\partial M\setminus 0$ as 
follows: With $A_+=r_+\op(a)e_+$ as above, 
  $$\sigma^\mu_\partial(A_+)(x^\prime,\xi^\prime)=a_{(\mu)}(x^\prime,0,\xi^\prime,D_{x_n}),$$
where $a_{(\mu)}$ is the principal component of $a$. For a trace operator $T$ as above,
  $$\sigma^\mu_\partial(T)(x^\prime,\xi^\prime)=
      \sum\limits_{j=0}^{d-1} s_{j,(\mu-j)}(x^\prime,\xi^\prime) r_{j} + 
      \sigma^\mu_\partial(T_0)(x^\prime,\xi^\prime),$$
where $S_j=\op(s_j)$, $r_ju=d^ju/dx_n^j(0)$, and 
  $$\sigma^\mu_\partial(T_0)(x^\prime,\xi^\prime)u=
      \int_{0}^\infty t(x^\prime,\xi^\prime;|\xi^\prime|x_n)u(x_n)\,dx_{n}.$$ 
For a Poisson operator $K$ as above,
  $$\sigma^\mu_\partial(K)(x^\prime,\xi^\prime)c
    =\big(x_n\mapsto k(x^\prime,\xi^\prime;|\xi^\prime|x_n)c\big)$$ 
$($mapping $c$ to a function of $x_n)$ and for a singular Green operator $G$ as above, 
  $$\sigma^\mu_\partial(G)(x^\prime,\xi^\prime)=
      \sum\limits_{j=0}^{d-1} \sigma^{\mu-j}(K_{j})(x^\prime,\xi^\prime) r_{j} + 
      \sigma^\mu_\partial(G_0)(x^\prime,\xi^\prime)$$
with 
  $$\big(\sigma^\mu_\partial(G_0)(x^\prime,\xi^\prime)u\big)(x_n)=
      \int_{0}^\infty g(x^\prime,\xi^\prime;|\xi^\prime|x_n,|\xi^\prime|y_n)u(y_n)\,dy_{n}.$$ 

\subsection{Operators on manifolds}\label{sec:6.2}
Via a partition of unity and local co-ordinates the calculus from the half-space as described above can be extended to compact manifolds $M$ with smooth boundary. We obtain classes $\calB^{\mu,d}(M;(E_0,J_0),(E_1,J_1))$ where $E_j$ and $J_j$ are vector bundles over $M$ and $\partial M$, respectively. 

Any $\calA\in\calB^{\mu,d}(M;(E_0,J_0),(E_1,J_1))$ induces continuous maps analogous 
to \eqref{eq:map}. We have the homogeneous principal symbol 
 $$\sigma_\psi^\mu(\calA):\; \pi^*E_0\lra\pi^*E_1$$
acting between the pull-backs of $E_j$ to the unit co-sphere bundle of $M$, and the boundary 
symbol 
 $$\sigma_\partial^\mu(\calA):\; 
     \pi_\partial^*\begin{pmatrix}E_0^\prime\otimes H^s(\rz_+)\\ \oplus \\ J_0\end{pmatrix}
     \lra
     \pi_\partial^*\begin{pmatrix}E_1^\prime\otimes H^{s-\mu}(\rz_+)\\ \oplus \\ 
     J_1\end{pmatrix},$$
where $\pi_\partial$ denotes the canonical projection of the unit co-sphere bundle of 
$\partial M$ to $\partial M$ and $E_j^\prime$ denotes the restriction of $E_j$ to the boundary. 

The composition of operators induces a map 
\begin{align}\label{eq:comp}
\begin{split}
 \calB^{\mu_1,d_1}&(M;(E_1,J_1),(E_2,J_2))\times \calB^{\mu_0,d_0}(M;(E_0,J_0),(E_1,J_1))\\
 &\lra \calB^{\mu_0+\mu_1,d}(M;(E_0,J_0),(E_2,J_2)),
 \qquad d=\max(d_0,d_1+\mu_0).
\end{split}
\end{align}
Both homogeneous principal symbol and principal boundary symbol are multiplicative under 
composition. 

The equivalence of ellipticity and Fredholm property in Boutet de Monvel's algebra have been shown 
in \cite{ReSc}, the spectral invariance in \cite{Schu54}. 

\subsection{Toeplitz subalgebras}\label{sec:6.3}
Let $\calP_j\in \calB^{0,0}(M;(E_j,J_j),(E_j,J_j))$, $j=0,1$, be two projections and set 
 $$ \calT^{\mu,d}(M;(E_0,J_0,P_0),(E_1,J_1,P_1))=
    \calP_1\,\calB^{\mu,d}(M;(E_j,J_j),(E_j,J_j))\,\calP_0.$$
Note that, according to \eqref{eq:comp}, 
 $$ \calT^{\mu,d}(M;(E_0,J_0,\calP_0),(E_1,J_1,\calP_1))\subset 
    \calB^{\mu,\max(\mu,d)}(M;(E_0,J_0),(E_1,J_1)). $$
We shall thus focus on the case $\mu\le d$. 
Since both $\sigma_\psi^0(\calP_j)$ and $\sigma_\partial^0(\calP_j)$ are projections we 
obtain subbundles 
\begin{align*}
 E^\psi_j(\calP_j)&:=\sigma_\psi^0(\calP_j)\big(\pi^*E_j\big)\subset \pi^*E_j,\\
 E^{\partial,s}_j(\calP_j)&:=\sigma_\partial^0(\calP_j)
  \left(
   \pi_\partial^*\begin{pmatrix}E_j^\prime\otimes H^s(\rz_+)\\ \oplus \\ J_j\end{pmatrix}
  \right)
  \subset\pi_\partial^*\begin{pmatrix}E_j^\prime\otimes H^s(\rz_+)\\ \oplus \\ J_j\end{pmatrix}. 
\end{align*}
The principal symbol ${\sigma}(\calA,\calP_0,\calP_1)$ consists of the two components 
\begin{align*}
 {\sigma}^\mu_\psi(\calA,\calP_0,\calP_1)&:\,E^\psi_0(\calP_0)\lra E^\psi_1(\calP_1),\\ 
 {\sigma}^\mu_\partial(\calA,\calP_0,\calP_1)&:\,E_0^{\partial,s}(\calP_0)
  \lra E_1^{\partial,s-\mu}(\calP_1),
\end{align*}
induced by the restriction of $\sigma_\psi^\mu(\calA)$ and $\sigma_\partial^\mu(\calA)$, 
respectively $($the choice of $s>d-1/2$ is arbitrary$)$. 
In the following let us write $t_+=\max(0,t)$. 

\begin{theorem}
For $\calA\in \calT^{\mu,\mu_+}(M;(E_0,J_0,\calP_0),(E_1,J_1,\calP_1))$ the following 
statements are equivalent: 
\begin{itemize}
 \item[$($a$)$] $\calA$ is elliptic, i.e., has a parametrix 
   $$\calB\in \calT^{-\mu,(-\mu)_+}(M;(E_1,J_1,\calP_1),(E_0,J_0,\calP_0)).\footnotemark$$   
 \item[$($b$)$] For\footnotetext{This means 
  \begin{align*}
   \calB\calA-\calP_0&\in \calT^{-\infty,\mu_+}(M;(E_0,J_0,\calP_0),(E_0,J_0,\calP_0))\\ 
   \calA\calB-\calP_1&\in \calT^{-\infty,(-\mu)_+}(M;(E_1,J_1,\calP_1),(E_1,J_1,\calP_1)).
  \end{align*}} 
  some $s>\mu_+-1/2$
   $$\calA:\;\calP_0\begin{pmatrix} H^s(M,E_0)\\ \oplus\\ H^s(\partial M,J_0)\end{pmatrix}
     \longrightarrow
     \calP_1\begin{pmatrix}H^{s-\mu}(M,E_1)\\ \oplus\\ H^{s-\mu}(\partial M,J_1)\end{pmatrix}$$
  is a Fredholm operator. 
 \item[$($c$)$] Both morphisms ${\sigma}^\mu_\psi(\calA,\calP_0,\calP_1)$ and 
  ${\sigma}^\mu_\partial(\calA,\calP_0,\calP_1)$ are isomorphisms. 
\end{itemize}
In this case, $($b$)$ is true for arbitrary $s>\mu_+-1/2$.  
\end{theorem}

Note that this result is both a generalization and strengthening of Theorem 2.2 of 
\cite{Schu37} in the case $d=\mu_+$. Also here the crucial point is to reduce to the 
case of $\mu=s=0$. This is then a particular case of Theorems \ref{thm:fred} and 
\ref{thm:sigma}, with 
\begin{align*}
 L^{\mu}(\bfg)&=\calB^{\mu,0}(M;(E_0,J_0),(E_1,J_1)),\qquad 
 \bfg=\big((M,E_0,J_0),(M,E_1,J_1)\big)\\
 H^s(g)&=
 \begin{matrix}H^s(M,E)\\ \oplus\\ H^s(\partial M,J)\end{matrix}, 
 \qquad g=(M,E,J).
\end{align*}
In fact, this reduction is possible, since one can show that for any bundles $E$ and $J$ 
there exist elements $\calR^\mu\in \calB^{\mu,0}(M;(E,J),(E,J))$, $\mu\in\gz$, that induce 
isomorphisms 
$H^s(M,E) \oplus H^s(\partial M,J)\to H^{s-\mu}(M,E) \oplus H^{s-\mu}(\partial M,J)$ and 
such that $(\calR^\mu)^{-1}=\calR_{-\mu}$; see \cite{Bout1}, \cite{ReSc}, \cite{Grub}. 
This also allows to obtain the spectral invariance, using Theorem \ref{thm:spectral}.  

\begin{theorem}\label{thm:inv}
If $\calA\in \calT^{\mu,\mu_+}(M;(E_0,J_0,\calP_0),(E_1,J_1,\calP_1))$ induces an isomorphism 
 $$\calA:\;\calP_0\begin{pmatrix} H^s(M,E_0)\\ \oplus\\ H^s(\partial M,J_0)\end{pmatrix}
   \longrightarrow
   \calP_1\begin{pmatrix}H^{s-\mu}(M,E_1)\\ \oplus\\ H^{s-\mu}(\partial M,J_1)\end{pmatrix}$$
for some $s>\mu_+-1/2$, then for all $s>\mu_+-1/2$ and 
   $$\calA^{-1}\in \calT^{-\mu,(-\mu)_+}(M;(E_1,J_1,\calP_1),(E_0,J_0,\calP_0)).$$ 
\end{theorem}

\subsection{The Stokes operator}\label{sec:6.4}

We shall now apply the results of the previous subsection to the Stokes operator. We shall 
need the following result on the existence of reductions of orders in Toeplitz subalgebras. 

\begin{lemma}\label{lem:red}
Let $Y$ be a closed Riemannian manifold and $E$ a hermitian vector bundle over $Y$. 
Moreover, let $P\in L^0_\cl(Y;E,E)$ be an orthogonal projection and $\mu\in\rz$. 
Then there exist $R^t\in T^t(Y;(E,P),(E,P))$ for $t=\mu$ and $t=-\mu$ such that 
$R^\mu R^{-\mu}=R^{-\mu}R^{\mu}=P$. 
\end{lemma}
\begin{proof}
We can assume $\mu>0$. Choose an $S\in L^\mu_\cl(Y;E,E)$ which is invertible, 
symmetric, and satisfies  
 $$(Su,u)>0,\qquad \text{for all }u\in\scrC^\infty(Y,E),$$
where $(\cdot,\cdot)$ is a scalar-product of $L^2(Y,E)$. Then 
$R:=PSP+(1-P)S(1-P)$ is also symmetric and 
 $$(Ru,u)=(SPu,Pu)+(S(1-P)u,(1-P)u)>0,\qquad \text{for all }u\in\scrC^\infty(Y,E).$$
Since the spectrum of elliptic and positiv operators consists of isolated positiv eigenvalues, 
we conclude that $R$ is invertible with inverse in $L^{-\mu}_\cl(Y;E,E)$. However, then 
$R^\mu:=PSP$ induces isomorphisms $H^s(Y,E,P)\to H^{s-\mu}(Y,E,P)$. Due to spectral invariance, 
cf.\ Theorem \ref{thm:spectral}, the claim follows. 
\end{proof}

Now let $n\ge2$ be the dimension of $M$ and let  
 $$L^2_\sigma(M,\cz^n)=\Big\{u\in L^2(M,\cz^n)\st \mathrm{div}\,u=0,\;\gamma_\nu u=0\Big\}$$
denote the space of square integrable solenoidal vector fields; here we use the notation 
 $$\gamma u=u|_{\partial M},\qquad \gamma_\nu u=\nu\cdot\gamma u,$$
where $\nu$ denotes the outer normal of $M$. Also let us set 
 $$H^s_\sigma(M,\cz^n)=H^s(M,\cz^n)\cap L^2_\sigma(M,\cz^n).$$
It has been shown in \cite{GrSo} that there is a projection 
$P\in \calB^{0,0}(M;(\cz^n,0),(\cz^n,0))$ 
$($the Helmholtz projection, of course$)$ such that 
 $$H^s_\sigma(M,\cz^n)=P\big(H^s(M,\cz^n)\big).$$
Let us define the projection $Q\in L^0_\cl(\partial M;\cz^n,\cz^n)$ by 
\begin{equation}\label{eq:pr}
 Qv=v-(v\cdot\nu)\nu.
\end{equation} 
This induces the orthogonal projection of $H^s(\partial M,\cz^n)$ onto 
 $$H^s_\nu(\partial M,\cz^n):=\big\{v\in H^s(\partial M,\cz^n)\st v\cdot\nu=0\big\}.$$
This space arises by restricting solenoidal vector fields to the boundary, i.e., 
\begin{equation}\label{eq:stokes0}
 \gamma:H^s_\sigma(M,\cz^n) \lra H^s_\nu(\partial M,\cz^n)
\end{equation} 
surjectively, cf.\ Proposition 2.1 in \cite{Giga}. 
The Stokes operator $($with Dirichlet boundary conditions$)$ is now 
\begin{equation}\label{eq:stokes1}
 \begin{pmatrix}P\Delta\\ \gamma\end{pmatrix}:\;
 H^{s}_\sigma(M,\cz^n)
 \lra
 \begin{matrix}H^{s-2}_\sigma(M,\cz^n)\\ \oplus\\ H^{s-1/2}_\nu(M,\cz^n)\end{matrix}.
\end{equation} 
Due to Lemma \ref{lem:red} we can choose an $R\in T^{3/2}(\partial M;\cz^n,\cz^n)$ 
inducing isomorphisms $H^{s}_\nu(M,\cz^n)\to H^{s-3/2}_\nu(M,\cz^n)$ for all $s$, 
and then consider 
\begin{equation}\label{eq:stokes2}
 \begin{pmatrix}P\Delta\\ T\end{pmatrix}:=
 \begin{pmatrix}1&0\\0&R\end{pmatrix}
 \begin{pmatrix}P\Delta\\ \gamma\end{pmatrix}:\;
 H^{s}_\sigma(M,\cz^n)
 \lra
 \begin{matrix}H^{s-2}_\sigma(M,\cz^n)\\ \oplus\\ H^{s-2}_\nu(M,\cz^n)\end{matrix}.
\end{equation} 
Now we can rewrite 
\eqref{eq:stokes2} in the form 
\begin{equation}\label{eq:stokes3}
 \calA:=\begin{pmatrix}P&0\\ 0&Q\end{pmatrix}
 \begin{pmatrix}\Delta\\ T\end{pmatrix}P. 
\end{equation} 
Setting 
\begin{align*}
 \calP_0&=P\in B^{0,0}(M;(\cz^n,0),\cz^n,0)),\\
 \calP_1&=\begin{pmatrix}P&0\\ 0&Q\end{pmatrix}\in \calB^{0,0}(M;(\cz^n,\cz^n),(\cz^n,\cz^n)),
\end{align*}
we obtain that $\calA$ belongs to a Toeplitz subalgebra, 
 $$\calA\in \calT^{2,2}(M;(\cz^n,0,\calP_0),(\cz^n,\cz^n,\calP_1)).$$
Now \eqref{eq:stokes1} is an isomorphism for $s=2$, hence so is \eqref{eq:stokes2}. In fact, \eqref{eq:stokes0} is surjective and the Dirichlet realization \eqref{eq:st} is known to be self-adjoint and positive, hence is an isomorphism. Then we can use the elementray fact that a linear map 
 $$\begin{pmatrix}A\\T\end{pmatrix}:H\lra \begin{matrix}E\\ \oplus \\ F\end{matrix}$$
is an isomorphism if, and only if, $T:H\to F$ is surjective and 
 $A:\mathrm{ker}\,T\to E$
is an isomorphism.

From Theorem \ref{thm:inv} we can conclude that the 
invertibility of $\calA$ holds for any $s>3/2$ and that the inverse is realized by an element 
of $\calT^{-2,0}(M;(\cz^n,\cz^n,\calP_1),(\cz^n,0,\calP_0))$. Reformulating this result for the 
original Stokes operator \eqref{eq:stokes1} yields the following: 

\begin{theorem}
The Stokes operator with Dirichlet boundary conditions,   
\begin{equation*}
 \begin{pmatrix}P\Delta\\ \gamma\end{pmatrix}:\;
 H^{s}_\sigma(M,\cz^n)
 \lra
 \begin{matrix}H^{s-2}_\sigma(M,\cz^n)\\ \oplus\\ H^{s-1/2}_\nu(M,\cz^n)\end{matrix},
\end{equation*} 
is invertible for any $s>3/2$ and 
 $$\begin{pmatrix}P\Delta\\ \gamma\end{pmatrix}^{-1}=
   \begin{pmatrix}P(A_++G)P & PKSQ\end{pmatrix},$$
where $P$ is the Helmholtz projection, $Q$ the orthogonal projection from \eqref{eq:pr},  
$S\in L^{-3/2}_\cl(\partial M;\cz^n,\cz^n)$, and 
 $$\begin{pmatrix}A_++G & K\end{pmatrix}\in\calB^{-2,0}(M;(\cz^n,\cz^n),(\cz^n,0)).$$
\end{theorem}

Since operators from Boutet de Monvel's algebra are known to act continuously in $L_p$-Sobolev and Besov spaces, cf.\ \cite{GrKo}, the previous theorem remains valid in a corresponding $L_p$-version with $1<p<\infty$. 

\subsection*{Acknowledgements}
The author thanks Robert Denk (Konstanz) and J\"urgen Saal (Darmstadt) for helpful discussions concerning the Stokes operator.

\begin{small}
\bibliographystyle{amsalpha}

\end{small}

\end{document}